\documentclass[11pt]{amsart}
\usepackage{amsfonts}
\usepackage{amsmath}

\newtheorem{theorem}{Theorem}
\theoremstyle{plain}

\newtheorem{lemma}{Lemma}

\newtheorem{proposition}{Proposition}

\numberwithin{equation}{section}
\input{tcilatex}

\begin{document}
\title[Singular elliptic equation involving the GJMS operator...]{Singular
elliptic equation involving the GJMS operator on compact Riemannian manifolds%
}
\author{Mohammed Benalili}
\address{Dept. Maths, Faculty of Sciences, University UABB, Tlemcen, Algeria}
\email{[m\_benalili@mail.univ-tlemcen.dz}
\author{Ali Zouaoui}
\email{2014zouaoui@gmail.com}

\begin{abstract}
In this paper we consider a singular elliptic equation involving the GJMS
(Graham-Jenne-Mason-Sparling) operator of order $k$ on $n$-dimensional
compact Riemannian manifold with $2k<n$. Mutiplicity and nonexistence
results are established.
\end{abstract}

\maketitle

Let $(M,g)$ be an $n$-dimensional Riemannian manifold. The $k$-th GJMS
operator (Graham-Jenne-Mason-Sparling, see (\cite{5}) $P_{g}$ is a
differential operator defined for any integer $k$ if the dimension n is odd,
and $2k\leq n$ otherwise. In the following, we will consider the case $%
2k\leq n$. $P_{g}$ is of the form

\begin{equation*}
P_{g}=\Delta ^{k}+lot
\end{equation*}%
where $\Delta =-div_{g}\left( \nabla \right) $ is the Laplacian-Beltrami
operator \ and $lot$ denotes the lower terms. One of the fundamental
property of $P_{g}$ is its behavior with respect to conformal change of
metrics: for $\varphi \in C^{\infty }(M)$ , $\varphi >0$ and $\overline{g}%
=\varphi ^{\frac{4}{n-2k}}g$ a conformal metric to $g$,$\ \ \ \ \ \ \ \ \ \
\ \ \ \ \ \ \ \ \ \ \ \ \ \ \ \ \ \ \ \ \ \ \ \ \ \ \ \ \ \ \ \ \ \ \ \ $%
\begin{equation}
\varphi ^{\frac{n+2k}{n-2k}}P_{\widetilde{g}}u=P_{g}\left( \varphi u\right) .
\label{0.0}
\end{equation}%
$P_{g}$ is self-adjoint with respect to the $L^{2}$-scalar . To $P_{g}$ is
associated a conformal invariant scalar function denoted $Q_{g}$ and is
called the $Q$-curvature. For $k=1$, the GJMS operator is ( up to a constant
) the conformal Laplacian and the corresponding $Q$-curvature function is
simply the scalar curvature. For $k=2$, the GJMS operator is the Paneitz
operator introduced in (\cite{13}). For $2k<n$, the Q-curvature is $Q_{g}=%
\frac{2}{n-2k}P_{g}(1)$. Many works was devoted the $Q$-curvature equation
in the last two decades (see \cite{2}, \cite{3}, \cite{4}, \cite{5}, \cite{7}%
, \cite{9}, \cite{13}, \cite{17}). Many authors investigated the
interactions of conformal methods with mathematical physic which led them to
study the Einstein-scalar fields Lichnerowiz equations (see \cite{6}, \cite%
{8}, \cite{12}, \cite{14}, \cite{15}, \cite{16}). These methods have been
extended to scalar fields Einstein-Licherowicz type equation involving the
Paneitz operator, (see \cite{9}). In this work we analyze an
Einstein-Lichnerowicz scalar field equation containing the $k$-th order GJMS
operator on a Riemannian $n$-dimensional manifold with $2k<n$; more
precisely we consider the following equation

\begin{equation}
\left\{ 
\begin{array}{c}
P_{g}(u)=B\left( x\right) u^{2^{\sharp }-1}+\frac{A\left( x\right) }{%
u^{2^{\sharp }+1}}+\frac{C(x)}{u^{p}} \\ 
u>0%
\end{array}%
\right.  \label{0.1}
\end{equation}%
where $2^{\sharp }=\frac{2n}{n-2k}$ and $p>1$. In all the sequel of this
paper we assume that the operator $P_{g}$ is coercive which allows us ( see
Proposition 2, \cite{17}) to endow $H_{k}^{2}\left( M\right) $ with the
following appropriated equivalent norm 
\begin{equation}
\Vert u\Vert =\sqrt{\int\limits_{M}u.P_{g}(u)dv_{g}}\text{.}  \label{0.2}
\end{equation}%
So we deduce from the coercivity of $P_{g}$ and the continuity of the
inclusion $H_{k}^{2}\left( M\right) \subset L^{2^{\sharp }}\left( M\right) $%
, the existence of a constant $S>0$ such that 
\begin{equation}
\Vert u\Vert _{2^{\sharp }}^{2^{\sharp }}\leq S\Vert u\Vert ^{2^{\sharp }}
\label{0.3}
\end{equation}%
where $2^{\sharp }=\frac{2n}{n-2k}$.

Our work is organized as follows: in a first section we show the existence
of a solution to equation (\ref{0.1}) obtained by means of the mountain-pass
theorem: more precisely we establish the following theorem.

\begin{theorem}
\label{theorem1} Let $\left( M,g\right) $ be a compact Riemannian manifold
with dimension $n>2k$ and $A>0,$ $B>0$, $C>0$ are smooth functions on $M$.
Suppose moreover that the operator\ $P_{g}$ is coercive and have a positive
Green function. If there exists a constant $C\left( n,p,k\right) >0$
depending only on $n,$ $p$ , $k$ such that 
\begin{equation}
\frac{\Vert \varphi \Vert ^{2^{\sharp }}}{2^{\sharp }}\int_{M}\frac{A(x)}{%
\varphi ^{2^{\natural }}}dv_{g}\leq C\left( n,p,k\right) \left( S\underset{%
x\in M}{\max }B(x)\right) ^{\frac{2+2^{\sharp }}{2-2^{\sharp }}}  \label{0.4}
\end{equation}%
\begin{equation}
\frac{\Vert \varphi \Vert ^{p-1}}{p-1}\int_{M}\frac{C(x)}{\varphi ^{p-1}}%
dv_{g}\leq C\left( n,p,k\right) \left( S\underset{x\in M}{\max }B(x)\right)
^{\frac{p+1}{2-2^{\sharp }}}  \label{0.5}
\end{equation}%
for some smooth function $\varphi >0$, then equation (\ref{0.1}) admits a
smooth solution.
\end{theorem}

In the second section we prove, by means of the Ekeland's lemma, the
existence of a second solution to equation (\ref{0.1}). In particular, by
setting $t_{0}=\left( \frac{1}{S\underset{x\in M}{\max }B(x)}\right) ^{\frac{%
n-2k}{4k}}$and $a=\frac{1}{\left( 2\left( n-k\right) \right) ^{\frac{%
2^{\natural }}{2}}}$ we obtain the following theorem:

\begin{theorem}
\label{theorem2} Let $(M,g)$ be a compact Riemannian manifold of dimension $%
n>2k,\;(k\in \mathbb{N}^{\ast })$. Suppose that the operator $P_{g}$ is
coercive; has a Green positive function and there is a constant $C(n,p,k)>0$
which depends only on $n,p,k$ such that: 
\begin{equation}
\frac{\Vert \varphi \Vert ^{2^{\sharp }}}{2^{\sharp }}\int_{M}\frac{A(x)}{%
\varphi ^{2^{\natural }}}dv_{g}\leq C\left( n,p,k\right) \left( S\underset{%
x\in M}{\max }B(x)\right) ^{\frac{2+2^{\sharp }}{2-2^{\sharp }}}  \label{0.6}
\end{equation}%
and 
\begin{equation}
\frac{\Vert \varphi \Vert ^{p-1}}{p-1}\int_{M}\frac{C(x)}{\varphi ^{p-1}}%
dv_{g}\leq C\left( n,p,k\right) \left( S\underset{x\in M}{\max }B(x)\right)
^{\frac{p+1}{2-2^{\sharp }}}  \label{0.7}
\end{equation}%
for some smooth function $\varphi >0$ on $M$. If moreover for every $%
\varepsilon \in \left] 0,\lambda ^{\ast }\right[ $ where $\lambda ^{\ast }$
is a positive constant, the following conditions occur 
\begin{equation*}
0<\left( S\underset{x\in M}{\max }B(x)\right) <\frac{a}{4}
\end{equation*}%
\begin{equation}
\left( \int\limits_{M}^{{}}\sqrt{A(x)}dv_{g}\right) ^{2}\left( \frac{8-a}{a}%
\right) \frac{(n-2k)Q_{\tilde{g}}}{2t_{0}^{2^{\sharp }}}>2^{\sharp }\frac{k}{%
n}t_{0}^{2}(1-\frac{a}{8})  \label{0.8}
\end{equation}%
and 
\begin{equation}
\int_{M}Q_{g}dv_{g}\neq k(n-1)\omega _{n}  \label{0.9}
\end{equation}%
where $\omega _{n}$ is the volume of the round sphere, $2^{\sharp }=\frac{2n%
}{n-2k}\;,3<p<2^{\sharp }+1$. Then the equation (\ref{0.1}) admits a second
smooth solution.
\end{theorem}

In the last section we give a nonexistence result of solution. Mainly we
show the following result:

\begin{theorem}
\label{theorem3} Given $(M,g)$ a compact Riemannian manifold of dimension $%
n>2k,\;(k\in \mathbb{N}^{\ast })$ and $A,B,C$ are positive smooth functions
on $M$ and $2<p<2^{\sharp }+1$. Assume that 
\begin{equation}
C(n,p,k)\left( \frac{\int\limits_{M}\sqrt{B.C}dv_{g}}{\int\limits_{M}Bdv_{g}}%
\right) ^{2.\frac{2^{\sharp }}{p-1+2^{\sharp }}}\int%
\limits_{M}Bdv_{g}>(SR)^{2}  \label{0.10}
\end{equation}%
where $S,R$ are positive constants and 
\begin{equation*}
C(n,p,k)=\frac{2^{\sharp }+p-1}{p-1}\left( \frac{p-1}{2^{\sharp }}\right) ^{%
\frac{2^{\sharp }}{2^{\sharp }+p-1}}\text{.}
\end{equation*}%
Then the equation (\ref{0.1}) has no smooth positive solution $u$ with
energy $\Vert u\Vert _{H_{k}^{2}(M)}\leq R$.
\end{theorem}

\section{Existence of a first solution}

In this section we show the theorem \ref{theorem1}. Before starting the
proof, we first give an example of manifolds where the GJMS operator $Pg$
has a positive Green function.

\begin{proposition}
Suppose that the metric $g$ is Einstein with positive scalar curvature of
dimension $n>2k$, then the GJMS operator $P_{g}$ admits a Green positive
function.
\end{proposition}

\begin{proof}
On $n$-dimensional Einstein manifold, the GJMS operator of order $k$ is
given by (see \cite{5})%
\begin{equation*}
P_{g}=\dprod\limits_{l=1}^{k}\left( \Delta -c_{l}Sc\right)
\end{equation*}%
where $c_{l}=\frac{\left( n+2l-2\right) \left( n-2l\right) }{4n\left(
n-1\right) }$, $Sc$ stands for the scalar curvature. If the scalar curvature
is positive it is well known that the operator $\Delta -c_{l}Sc$ has a
positive Green function. Denote by $L_{l}=\Delta -c_{l}Sc$, $l=1,...,k$; by
definition of the Green function of $L_{l}$ we know that for all $u\in
C^{\infty }(M)$, 
\begin{equation*}
\left( L_{l}u\right) \left( x\right) =\int_{M}G_{l+1}((x,y)\left(
L_{l+1}L_{l}u\right) (y)dv_{g}\left( y\right) .
\end{equation*}%
So 
\begin{equation*}
u(x)=\int_{M}G_{l}(x,z)\left( L_{l}u\right) \left( z\right) dv_{g}(z)+\frac{1%
}{Vol(M)}\int_{M}u(x)dv_{g}(x)
\end{equation*}%
\begin{equation*}
=\int_{M}\left( \int_{M}G_{l}(x,z)G_{l+1}((z,y)dv_{g}(z)\right) \left(
L_{l+1}L_{l}u\right) (y)dv_{g}\left( y\right) +\frac{1}{Vol(M)}%
\int_{M}u(x)dv_{g}(x)
\end{equation*}%
and letting%
\begin{equation*}
G_{l,l+1}(x,y)=G_{l}\ast
G_{l+1}(x,y)=\int_{M}G_{l}(x,z)G_{l+1}((z,y)dv_{g}(z).
\end{equation*}%
By induction, we get%
\begin{equation*}
u(x)=\int_{M}G_{1}\ast ...\ast G_{k}(x,y)P_{g}(y)dv_{g}\left( y\right) .
\end{equation*}%
Thus $P_{g}$ admits a positive Green function.
\end{proof}

To show the existence of solutions to equation (\ref{0.1}), we follow the
strategy in the proof of the paper by Hebey-Pacard-Pollack \cite{6}. We
consider the following $\epsilon $-approximating equations ( $\varepsilon >0$
)%
\begin{equation}
P_{g}(u)=B\left( x\right) \left( u^{+}\right) ^{2^{\sharp }-1}+\frac{A\left(
x\right) u^{+}}{\left( \varepsilon +\left( u^{+}\right) ^{2}\right)
^{2^{\flat }+1}}+\frac{C(x)u^{+}}{\left( \varepsilon +\left( u^{+}\right)
^{2}\right) ^{\frac{p+1}{2}}}  \label{2.1}
\end{equation}%
where $2^{\flat }=\frac{2^{\sharp }}{2}$, $p>1.$Which gives us a sequence $%
\left( u_{\varepsilon }\right) _{\varepsilon }$ of solutions to (\ref{2.1}).
The solution of equation (\ref{0.1}) is then obtained as the limiting of $%
\left( u_{\varepsilon }\right) _{\varepsilon }$, when $\varepsilon
\rightarrow 0$. To get rid of negative exponents, we consider the energy
functional associated to (\ref{2.1}) defined by, for any $\varepsilon >0$ 
\begin{equation*}
I_{\varepsilon }\left( u\right) =I^{(1)}\left( u\right) +I_{\varepsilon
}^{\left( 2\right) }\left( u\right)
\end{equation*}%
\ where $I^{\left( 1\right) }:H_{k}^{2}\left( M\right) \rightarrow \mathbb{R}
$ is given by 
\begin{equation*}
I^{\left( 1\right) }\left( u\right) =\frac{1}{2}\int_{M}uP_{g}(u)dv_{g}-%
\frac{1}{2^{\sharp }}\int_{M}B\left( x\right) \left( u^{+}\right)
^{2^{\sharp }}dv_{g}
\end{equation*}%
and \ $I_{\varepsilon }^{\left( 2\right) }:$\ $H_{k}^{\text{ }2}\left(
M\right) \longrightarrow \mathbb{R}$\ is%
\begin{equation*}
I_{\varepsilon }^{\left( 2\right) }\left( u\right) =\frac{1}{2^{\sharp }}%
\int_{M}\frac{A\left( x\right) }{\left( \varepsilon +\left( u^{+}\right)
^{2}\right) ^{2^{\flat }}}dv_{g}+\frac{1}{p-1}\int_{M}\frac{C(x)}{\left(
\varepsilon +\left( u^{+}\right) ^{2}\right) ^{\frac{p-1}{2}}}dv_{g}\text{.}
\end{equation*}%
It is easy to check the following inequality

\begin{equation}
\Phi \left( \Vert u\Vert \right) \leq I^{\left( 1\right) }\left( u\right)
\leq \Psi \left( \left\Vert u\right\Vert \right)   \label{2.2}
\end{equation}%
with%
\begin{equation*}
\Phi \left( t\right) =\frac{1}{2}t^{2}-\frac{1}{2^{\sharp }}\left(
S\max_{M}\left\vert B\right\vert \right) t^{2^{\sharp }}
\end{equation*}%
and%
\begin{equation*}
\Psi \left( t\right) =\frac{1}{2}t^{2}+\frac{1}{2^{\sharp }}\left(
S\max_{x\in M}B(x)\right) t^{2^{\sharp }}\text{.}
\end{equation*}%
The function $\Phi \left( t\right) $ is increasing on $\left[ 0,t_{0}\right] 
$ and decreasing on $\left] t_{0},+\infty \right[ $, where%
\begin{equation}
t_{0}=\left( \frac{1}{S\underset{x\in M}{\max }B(x)}\right) ^{\frac{n-2k}{4k}%
}  \label{2.3}
\end{equation}%
and 
\begin{equation}
\Phi \left( t_{0}\right) =\left( \frac{1}{2}-\frac{1}{2^{\sharp }}\right)
\left( \frac{1}{S\underset{x\in M}{\max }B(x)}\right) ^{\frac{n-2k}{2k}}=%
\frac{k}{n}t_{0}^{2}\text{.}  \label{2.4}
\end{equation}

\begin{lemma}
\label{lem1} Let $\theta >0$ such that 
\begin{equation*}
\left( \frac{a}{2}\right) ^{\frac{2}{2^{\natural }}}<\theta ^{2}<a^{\frac{2}{%
2^{\natural }}}
\end{equation*}%
where%
\begin{equation*}
a=\frac{1}{\left( 2\left( n-k\right) \right) ^{\frac{2^{\natural }}{2}}}
\end{equation*}%
and put%
\begin{equation*}
t_{1}=\theta t_{0}\text{.}
\end{equation*}%
Then we have the following double inequality%
\begin{equation}
\Psi \left( t_{1}\right) \leq \theta ^{2}\frac{2^{\sharp }+2}{2^{\sharp }-2}%
\Phi \left( t_{0}\right) <\frac{1}{2k}\Phi \left( t_{0}\right) \text{.}
\label{2.5}
\end{equation}
\end{lemma}

\begin{proof}
In fact 
\begin{eqnarray*}
\Psi \left( t_{1}\right)  &=&\frac{1}{2}t_{1}^{2}+S\underset{x\in M}{\max }%
B(x)\frac{t_{1}^{2^{\sharp }}}{2^{\sharp }} \\
&=&\theta ^{2}\left( \frac{1}{2}t_{0}^{2}+S\underset{x\in M}{\max }%
B(x)\theta ^{2^{\sharp }-2}\frac{t_{0}^{2^{\sharp }}}{2^{\sharp }}\right) 
\text{.}
\end{eqnarray*}%
Since 
\begin{equation}
t_{0}^{2^{\sharp }}\left( S\underset{x\in M}{\max }B(x)\right) =t_{0}^{2}
\label{2.6}
\end{equation}%
we get 
\begin{eqnarray*}
\Psi \left( t_{1}\right)  &=&\theta ^{2}\left( \frac{1}{2}t_{0}^{2}+\frac{%
\theta ^{2^{\sharp }-2}}{2^{\sharp }}t_{0}^{2}\right)  \\
&=&\theta ^{2}t_{0}^{2}\left[ \frac{1}{2}+\frac{\theta ^{2^{\sharp }-2}}{%
2^{\sharp }}\right]  \\
&\leq &\theta ^{2}t_{0}^{2}\left( \frac{1}{2}+\frac{n-2k}{2n}\right)  \\
&\leq &\left( 1-\frac{k}{n}\right) \theta ^{2}t_{0}^{2}.
\end{eqnarray*}%
And since%
\begin{equation*}
\frac{2^{\sharp }+2}{2^{\sharp }-2}=\left( \frac{n}{k}-1\right) \ \text{\
and \ }\Phi \left( t_{0}\right) =\frac{k}{n}t_{0}^{2}
\end{equation*}%
we infer that 
\begin{equation*}
\Psi \left( t_{1}\right) \leq \theta ^{2}\frac{2^{\sharp }+2}{2^{\sharp }-2}%
\Phi \left( t_{0}\right) <\frac{1}{2k}\Phi \left( t_{0}\right) \text{.}
\end{equation*}
\end{proof}

Now we check the Mountain-Pass lemma conditions for the functional $%
I_{\varepsilon }$.

\begin{lemma}
\label{lem2} The functional $I_{\epsilon }$ satisfies the following
condition: there exists an open ball $B(u_{1},\rho )$ of radius $\rho >0$
and of center some $u_{1}$ in $H_{k}^{2}\left( M\right) $ and there are $%
u_{2}\notin \overline{B}(u_{1},\rho )$ and a real number $c_{o}$\ such that 
\begin{equation*}
\max \left( I_{\epsilon }(u_{1}),I_{\epsilon }(u_{2})\right) <c_{o}\leq
I_{\epsilon }(u)
\end{equation*}%
for all $u\in \partial B(u_{1},\rho )$.
\end{lemma}

\begin{proof}
Following the strategy of the proof in the paper by Hebey-Pacard-Pollack 
\cite{6}, we let $\varphi \in C^{\infty }\left( M\right) $, $\varphi >0$ on $%
M$ and without loss of generality we may assume $\Vert \varphi \Vert =1$. Put%
\begin{equation}
C\left( n,p,k\right) =\left( 2k-1\right) \frac{\theta ^{2^{\sharp }+p}}{4n}%
\leq C_{1}\left( n,k\right) =\frac{\left( 2k-1\right) \theta ^{2^{\sharp }}}{%
4n}\text{.}  \label{2.10}
\end{equation}%
The inequality (\ref{0.4}) becomes%
\begin{equation}
\frac{1}{2^{\sharp }}\int_{M}\frac{A\left( x\right) }{\left( t_{1}\varphi
\right) ^{2^{\sharp }}}dv_{g}\leq \frac{2k-1}{4k}\Phi \left( t_{0}\right) 
\text{.}  \label{2.11}
\end{equation}%
Indeed, we have 
\begin{equation}
\Phi (t_{o})=\frac{k}{n}t_{o}^{2}\text{.}  \label{2.12}
\end{equation}%
and by (\ref{2.10}), we get%
\begin{equation*}
\frac{1}{2^{\natural }}\int_{M}\frac{A(x)}{\left( t_{1}\varphi \right)
^{2^{\sharp }}}dv_{g}\leq \frac{C_{1}\left( n,k\right) }{t_{o}^{2^{\natural
}}\theta ^{2^{\natural }}}\left( S.\max_{M}B\left( x\right) \right) ^{\frac{%
2+2^{\sharp }}{2-2^{\sharp }}}
\end{equation*}%
\begin{equation*}
=\frac{2k-1}{4k}\Phi (t_{o})\text{.}
\end{equation*}%
Analogously, by putting 
\begin{equation*}
C_{2}(n,p,k)=\frac{\left( 2k-1\right) \theta ^{p-1}}{4n}
\end{equation*}%
we obtain 
\begin{equation}
\frac{1}{p-1}\int_{M}\frac{C(x)}{\left( t_{1}\varphi \right) ^{p-1}}%
dv_{g}\leq \frac{2k-1}{4k}\Phi (t_{o})\text{.}  \label{2.13}
\end{equation}%
By relations (\ref{2.5}), (\ref{2.6}), (\ref{2.11}) and (\ref{2.13}), we
infer that 
\begin{equation*}
I_{\epsilon }\left( t_{1}\varphi \right) \leq \Psi \left( \Vert t_{1}\varphi
\Vert \right) +\frac{1}{2^{\sharp }}\int_{M}\frac{A\left( x\right) }{\left(
\varepsilon +\left( t_{1}\varphi \right) ^{2}\right) ^{2^{\flat }}}dv_{g}+%
\frac{1}{p-1}\int_{M}\frac{C(x)}{\left( \varepsilon +\left( t_{1}\varphi
\right) ^{2}\right) ^{\frac{p-1}{2}}}dv_{g}
\end{equation*}%
\begin{equation}
\leq \Psi \left( t_{1}\right) +\frac{1}{2^{\sharp }}\int_{M}\frac{A\left(
x\right) }{\left( \varepsilon +\left( t_{1}\varphi \right) ^{2}\right)
^{2^{\flat }}}dv_{g}+\frac{1}{p-1}\int_{M}\frac{C(x)}{\left( \varepsilon
+\left( t_{1}\varphi \right) ^{2}\right) ^{\frac{p-1}{2}}}dv_{g}\leq \Phi
\left( t_{0}\right) .  \label{2.14}
\end{equation}%
Again from (\ref{2.5}), we deduce that 
\begin{equation*}
I_{\epsilon }\left( t_{0}\varphi \right) \geq \Phi \left( t_{0}\right) +%
\frac{1}{2^{\sharp }}\int_{M}\frac{A\left( x\right) }{\left( \varepsilon
+\left( t_{0}\varphi \right) ^{2}\right) ^{2^{\flat }}}dv_{g}+\frac{1}{p-1}%
\int_{M}\frac{C(x)}{\left( \varepsilon +\left( t_{o}\varphi \right)
^{2}\right) ^{\frac{p-1}{2}}}dv_{g}
\end{equation*}%
and since $A$ and $C$ are assumed with positive values, we obtain%
\begin{equation}
I_{\epsilon }\left( t_{0}\varphi \right) \geq \Phi \left( t_{0}\right) \text{%
.}  \label{2.15}
\end{equation}%
Finally from (\ref{2.14}) and (\ref{2.15}), we get 
\begin{equation*}
I_{\epsilon }\left( t_{1}\varphi \right) <\Phi \left( t_{0}\right) \leq
I_{\epsilon }\left( t_{0}\varphi \right) .
\end{equation*}%
Noting that 
\begin{eqnarray*}
\underset{t\rightarrow +\infty }{\lim }I_{\epsilon }\left( t\varphi \right)
&=&\underset{t\rightarrow +\infty }{\lim }\left[ \frac{1}{2}\Vert t\varphi
\Vert _{P_{g}}^{2}-\frac{1}{2^{\sharp }}\int_{M}\left( B(x)\left( t\varphi
\right) ^{2^{\sharp }}dv_{g}-\frac{A(x)}{\left( \varepsilon +\left( t\varphi
\right) ^{2}\right) ^{2^{\flat }}}\right) dv_{g}\right] \\
&-&\underset{t\rightarrow +\infty }{\lim }\frac{1}{p-1}\int\limits_{M}\frac{%
C(x)}{\left( \varepsilon +\left( t\varphi \right) ^{2}\right) ^{\frac{p-1}{2}%
}}dv_{g} \\
&=&\underset{t\rightarrow +\infty }{\lim }t^{2^{\sharp }}\left( \frac{1}{%
2t^{2^{\sharp }-2}}-\frac{1}{2^{\sharp }}\int_{M}B\left( x\right) \varphi
^{2^{\sharp }}dv\left( g\right) \right)
\end{eqnarray*}%
and since\ $\int_{M}B(x)\varphi ^{2^{\sharp }}dv_{g}>0$, we obtain 
\begin{equation*}
\underset{t\rightarrow +\infty }{\lim }I_{\epsilon }\left( t\varphi \right)
=-\infty .
\end{equation*}%
Consequently there is $t_{2}$ such that 
\begin{equation*}
t_{2}>t_{0}\quad \text{and}\quad I_{\epsilon }\left( t_{2}\varphi \right) <0.
\end{equation*}%
Now, to have the conditions of Lemma \ref{lem2} fulfilled, we just put 
\begin{equation*}
\ \ u_{1}=t_{1}\varphi \text{, }u_{2}=t_{2}\varphi \text{, }u=t_{0}\varphi
\end{equation*}%
and we take $\rho =t_{0}-t_{1}>0$ and $c_{0}=\Phi \left( t_{0}\right) $.
\end{proof}

Lemma \ref{lem1} allows us to apply the Mountain-Pass Lemma to the
functional $I_{\epsilon }$. Let%
\begin{equation*}
C_{\epsilon }=\underset{\gamma \in \Gamma }{\inf }\underset{u\in \gamma }{%
\max }I_{\epsilon }\left( u\right)
\end{equation*}%
where $\Gamma $ denotes the set of paths in $H_{k}^{2}\left( M\right) $
joining the functions $u_{1}=t_{1}\varphi $ and $u_{2}=t_{2}\varphi $.%
\newline
So $C_{\varepsilon }$ is a critical value of $I_{\varepsilon }$ and moreover 
\begin{equation*}
C_{\epsilon }>\Phi \left( t_{0}\right)
\end{equation*}%
and by putting $\gamma \left( t\right) =t\varphi $, for $t\in \left[
t_{1},t_{2}\right] $, we see that $C_{\varepsilon }$ is uniformly bounded
when $\varepsilon $ goes to $0$, so we get%
\begin{equation}
0<\Phi \left( t_{0}\right) <C_{\epsilon }\leq C  \label{2.16}
\end{equation}%
for $\varepsilon $ sufficiently small and $C>0$ not depending on $%
\varepsilon .\newline
$Consequently there exists a sequence $\left( u_{m}\right) _{m}$ of
functions in $H_{k}^{2}\left( M\right) $ such that%
\begin{equation}
I_{\epsilon }\left( u_{m}\right) \underset{m\rightarrow +\infty }{%
\rightarrow }C_{\epsilon }\quad \text{and}\quad DI_{\epsilon }\left(
u_{m}\right) \underset{m\rightarrow +\infty }{\rightarrow }0  \label{2.17}
\end{equation}%
By Lemma \ref{lem2} the sequence $\left( u_{m}\right) _{m\in N}$ of $%
H_{k}^{2}\left( M\right) $ is a Palais-Smale sequence (P-S) for the
functional $I_{\epsilon }$.

\begin{theorem}
\label{theorem4} The Palais-Smale sequence $\left( u_{m}\right) _{m\in N}$
is bounded in $H_{k}^{2}\left( M\right) $ and converges weakly to nontrivial
smooth solution $u_{\varepsilon }$ of equation (\ref{2.1}).
\end{theorem}

\begin{proof}
By (\ref{2.17}) we get for any $\varphi \in H_{k}^{\text{ }2}\left( M\right) 
$%
\begin{equation*}
DI_{\epsilon }\left( u_{m}\right) \varphi =o\left( \left\Vert \varphi
\right\Vert \right)
\end{equation*}%
i.e. \ for any $\varphi \in H_{k}^{\text{ }2}\left( M\right) $ one has%
\begin{equation}
\int_{M}\varphi P_{g}u_{m}dv_{g}=\int_{M}B\left( x\right) \left(
u_{m}^{+}\right) ^{2^{\sharp }-1}\varphi dv_{g}  \label{2.18}
\end{equation}%
\begin{equation*}
+\int_{M}\frac{A\left( x\right) u_{m}^{+}\varphi }{\left( \varepsilon
+\left( u_{m}^{+}\right) ^{2}\right) ^{2^{\flat }+1}}dv_{g}+\int_{M}\frac{%
C(x)u_{m}^{+}\varphi }{\left( \varepsilon +\left( u_{m}^{+}\right)
^{2}\right) ^{\frac{p}{2}+1}}dv_{g}+o\left( \left\Vert u_{m}\right\Vert
\right)
\end{equation*}%
in particular, for $\varphi =u_{m}$ we have%
\begin{equation*}
\int_{M}u_{m}P_{g}u_{m}dv_{g}-\int_{M}B(x)\left( u_{m}^{+}\right)
^{2^{\sharp }}dv_{g}=\int_{M}\frac{A(x)\left( u_{m}^{+}\right) ^{2}}{\left(
\varepsilon +\left( u_{m}^{+}\right) ^{2}\right) ^{2^{\flat }+1}}dv_{g}
\end{equation*}%
\begin{equation*}
+\int_{M}\frac{C(x)\left( u_{m}^{+}\right) ^{2}}{\left( \varepsilon +\left(
u_{m}\right) ^{2}\right) ^{\frac{p}{2}+1}}dv_{g}+o\left( \left\Vert
u_{m}\right\Vert \right) \text{.}
\end{equation*}%
or 
\begin{equation*}
-\frac{1}{2}\int_{M}u_{m}P_{g}u_{m}dv_{g}+\frac{1}{2}\int_{M}B(x)\left(
u_{m}^{+}\right) ^{2}dv_{g}
\end{equation*}%
\begin{equation}
+\frac{1}{2}\int_{M}\frac{A(x)\left( u_{m}^{+}\right) ^{2}}{\left(
\varepsilon +\left( u_{m}^{+}\right) ^{2}\right) ^{2^{\flat }+1}}dv_{g}+%
\frac{1}{2}\int_{M}\frac{C(x)\left( u_{m}^{+}\right) ^{2}}{\left(
\varepsilon +\left( u_{m}^{+}\right) ^{2}\right) ^{\frac{p}{2}+1}}%
dv_{g}=o\left( \left\Vert u_{m}\right\Vert \right) \text{.}  \label{2.19}
\end{equation}%
On the other hand it comes from (\ref{2.17}) that%
\begin{equation*}
\frac{1}{2}\int_{M}u_{m}P_{g}u_{m}dv_{g}-\frac{1}{2^{\sharp }}%
\int_{M}B(x)\left( u_{m}^{+}\right) ^{2^{\sharp }}dv_{g}
\end{equation*}%
\begin{equation}
+\frac{1}{2^{\sharp }}\int_{M}\frac{A(x)}{\left( \varepsilon +\left(
u_{m}^{+}\right) ^{2}\right) ^{2^{\flat }}}dv_{g}+\frac{1}{p-1}\int_{M}\frac{%
C(x)}{\left( \varepsilon +\left( u_{m}^{+}\right) ^{2}\right) ^{\frac{p-1}{2}%
}}dv_{g}=C_{\epsilon }+o\left( \left\Vert u_{m}\right\Vert \right) \text{.}
\label{2.20}
\end{equation}%
So by adding (\ref{2.19}) and (\ref{2.20}) we get 
\begin{equation}
\frac{k}{n}\int_{M}B(x)\left( u_{m}^{+}\right) ^{2^{\sharp }}dv_{g}+\frac{1}{%
2}\int_{M}\frac{A(x)\left( u_{m}^{+}\right) ^{2}}{\left( \varepsilon +\left(
u_{m}^{+}\right) ^{2}\right) ^{2^{\flat }+1}}dv_{g}+\frac{1}{2^{\sharp }}%
\int_{M}\frac{A(x)}{\left( \varepsilon +\left( u_{m}^{+}\right) ^{2}\right)
^{2^{\flat }}}dv_{g}  \label{2.21}
\end{equation}%
\begin{equation*}
+\frac{1}{2}\int_{M}\frac{C(x)\left( u_{m}^{+}\right) ^{2}}{\left(
\varepsilon +\left( u_{m}^{+}\right) ^{2}\right) ^{\frac{p}{2}+1}}dv_{g}+%
\frac{1}{p-1}\int_{M}\frac{C(x)}{\left( \varepsilon +\left( u_{m}^{+}\right)
^{2}\right) ^{\frac{p-1}{2}}}dv_{g}=C_{\epsilon }+o\left( \left\Vert
u_{m}\right\Vert \right) \text{.}
\end{equation*}%
For sufficiently large $m$ we deduce that%
\begin{equation*}
\frac{k}{n}\int_{M}B(x)\left( u_{m}^{+}\right) ^{2^{\sharp }}dv_{g}\leq
2C_{\epsilon }+o\left( \left\Vert u_{m}\right\Vert \right)
\end{equation*}%
or%
\begin{equation*}
\frac{1}{2^{\sharp }}\int_{M}B(x)\left( u_{m}^{+}\right) ^{2^{\sharp
}}dv_{g}\leq \frac{n}{2^{\sharp }}C_{\epsilon }+o\left( \left\Vert
u_{m}\right\Vert \right)
\end{equation*}%
and plugging this last inequality with in (\ref{2.20}) we obtain 
\begin{equation*}
\frac{1}{2}\int_{M}u_{m}P_{g}u_{m}dv\left( g\right) \leq C_{\epsilon }+\frac{%
n}{2^{\sharp }}C_{\epsilon }+o\left( \left\Vert u_{m}\right\Vert \right)
\end{equation*}%
\begin{equation*}
\leq nC_{\epsilon }+\frac{n\left( n-2k\right) }{2n}C_{\epsilon }+o\left(
\left\Vert u_{m}\right\Vert\right) \leq 2\left( n-k\right) C_{\epsilon
}+o\left( \left\Vert u_{m}\right\Vert \right) \text{.}
\end{equation*}%
Hence for $m$ large enough%
\begin{equation*}
\int_{M}u_{m}P_{g}u_{m}dv\left( g\right) \leq 4nC_{\epsilon }+o\left(
1\right) \leq 4nC_{\epsilon }+1
\end{equation*}%
i.e.%
\begin{equation}
\left\Vert u_{m}\right\Vert ^{2}\leq 4nC_{\epsilon }+1  \label{2.22}
\end{equation}%
Thus we prove the sequence $\left( u_{m}\right) _{m}$ is bounded in $H_{k}^{%
\text{ }2}\left( M\right) $; so we can extract a subsequence, still denoted $%
\left( u_{m}\right) _{m}$ which verifies:

1. $u_{m}\rightarrow u_{\varepsilon }$ weakly in $H_{k}^{\text{ }2}\left(
M\right) .$

2. $u_{m}\rightarrow u_{\varepsilon }$ strongly in $L^{p}\left( M\right) $, $%
\forall p<\frac{2n}{n-2k}$

3. $u_{m}\rightarrow u_{\varepsilon }$ a.e. in $M.$

4. $\left( u_{m}\right) ^{2^{\sharp }-1}\rightarrow u_{\varepsilon
}^{2^{\sharp }-1}$ weakly in $L^{\frac{2^{\sharp }}{2^{\sharp }-1}}\left(
M\right) $. \newline
Furthermore, putting $g\left( x\right) =\frac{1}{\varepsilon ^{q}},$ where $%
\varepsilon >0$ and $q>0,$ we get by Lebesgue's dominated convergence
theorem that%
\begin{equation*}
\forall k\in 
%TCIMACRO{\U{2115} }%
%BeginExpansion
\mathbb{N}
%EndExpansion
\text{: }\left( \left( u_{m}^{+}\right) ^{2}+\varepsilon \right)
^{-q}<\varepsilon ^{-q}\text{ \ and \ }\varepsilon ^{-q}\in L^{p}\left(
M\right) \text{ \ }\forall p\geq 1
\end{equation*}%
thus $\left( \left( u_{m}^{+}\right) ^{2}+\varepsilon \right)
^{-q}\rightarrow \left( \left( u_{\epsilon }\right) ^{2}+\varepsilon \right)
^{-q}$\ strongly in $L^{p}\left( M\right) $ \ $\forall p\geq 1$ and with (2)
we infer that $\frac{u_{m}^{+}}{\left( \left( u_{m}^{+}\right)
^{2}+\varepsilon \right) ^{q}}\rightarrow \frac{u_{\epsilon }^{+}}{\left(
\varepsilon +\left( u_{\varepsilon }^{+}\right) ^{2}\right) ^{q}}$ \
strongly in $L^{2}\left( M\right) $. So if we let $m$ go to $+\infty $ in (%
\ref{2.18}) we obtain that $u_{\epsilon }$ is a weak solution of the equation%
\begin{equation}
P_{g}u_{\varepsilon }=B\left( x\right) \left( u_{\varepsilon }^{+}\right)
^{2^{\sharp }-1}+\frac{A\left( x\right) u_{\varepsilon }^{+}}{\left(
\varepsilon +\left( u_{\varepsilon }^{+}\right) ^{2}\right) ^{2^{\flat }+1}}+%
\frac{C(x)u_{\varepsilon }^{+}}{\left( \varepsilon +\left( u_{\varepsilon
}^{+}\right) ^{2}\right) ^{\frac{p+1}{2}}}  \label{2.23}
\end{equation}%
where $2^{b}=\frac{2^{\natural }}{2}$ and $p>1$.

Our solution $u_{\epsilon }$ is not identically zero: indeed by (\ref{2.21}%
), we have 
\begin{equation*}
\frac{1}{2^{\sharp }}\int_{M}\frac{A\left( x\right) }{\left( \varepsilon
+\left( u_{m}\right) ^{2}\right) ^{2^{\flat }}}dv\left( g\right) \leq
C_{\varepsilon }+o(\left\Vert u_{m}\right\Vert )\text{.}
\end{equation*}%
Now, letting $m\rightarrow +\infty $ and taking in mind (\ref{2.16}), we
infer that 
\begin{equation}
\frac{1}{2^{\sharp }}\int_{M}\frac{A\left( x\right) }{\left( \varepsilon
+\left( u_{\varepsilon }^{+}\right) ^{2}\right) ^{2^{\flat }}}dv\left(
g\right) \leq C  \label{2.24}
\end{equation}%
where $C$ is the upper bound of $C_{\varepsilon }$.\newline
Now if for a sequence $\underset{j\rightarrow +\infty }{\varepsilon
_{j}\rightarrow 0}$ ( with $\epsilon _{j}>0,$ $\forall j\in 
%TCIMACRO{\U{2115} }%
%BeginExpansion
\mathbb{N}
%EndExpansion
$ ); $u_{\varepsilon _{j}}$ goes to $0$, then it follows that%
\begin{equation}
\frac{1}{2^{\sharp }\left( 2^{\sharp }-1\right) \varepsilon _{j}^{2^{\flat }}%
}\int_{M}A\left( x\right) dv\left( g\right) \leq C\text{.}  \label{2.25}
\end{equation}%
So if $j$ $\rightarrow +\infty ,$ it leads to a contradiction since by
assumption $A>0$.

Finally, for sufficiently small $\varepsilon $, $u_{\varepsilon }$ is a
solution not identically zero of the equation (\ref{2.1}).

Now we will show the regularity of $u_{\varepsilon }$. First we write the
equation (\ref{2.23}) in the form

\begin{equation*}
P_{g}u_{\varepsilon }=b(x,u_{\varepsilon })u_{\varepsilon }
\end{equation*}%
where 
\begin{equation*}
b(x,u_{\varepsilon })=B\left( x\right) \left( u_{\varepsilon }^{+}\right)
^{2^{\sharp }-2}+\frac{A\left( x\right) }{\left( \varepsilon +\left(
u_{\varepsilon }^{+}\right) ^{2}\right) ^{2^{\flat }+1}}+\frac{C(x)}{\left(
\varepsilon +\left( u_{\varepsilon }^{+}\right) ^{2}\right) ^{\frac{p+1}{2}}}%
\text{.}
\end{equation*}%
Since $\frac{A}{\left( \varepsilon +\left( u_{\varepsilon }^{+}\right)
^{2}\right) ^{2^{\flat }+1}}+\frac{C}{\left( \varepsilon +\left(
u_{\varepsilon }^{+}\right) ^{2}\right) ^{\frac{p+1}{2}}}\in L^{\infty
}\left( M\right) $ \ and $u_{\varepsilon }\in H_{k}^{2}\left( M\right)
\subset L^{2^{\sharp }}\left( M\right) $, we infer that $b\in L^{\frac{n}{2k}%
}\left( M\right) $. By the work of S. Mazumdar ( see the proof of the
theorem 5 page 28 in \cite{10} ) we obtain that $u_{\varepsilon }\in
L^{p}(M) $ for any $0<p<+\infty $. According to \cite{1}, we obtain that $%
u_{\varepsilon }\in H_{k}^{p}\left( M\right) $ for all $1<p<+\infty $. By
the same arguments as in the proof of proposition 8.3 in \cite{1} we
conclude that $u_{\varepsilon }\in C^{2k,\alpha }(M)$ with $\alpha \in
\left( 0,1\right) $.
\end{proof}

Now we are in position to prove Theorem \ref{theorem1}.

\begin{proof}
From what precedes $u_{\epsilon }$ is a $C^{2k}\left( M\right) $ nontrivial
solution to equation (\ref{2.1}), moreover $u_{\epsilon }$ is a weak limit
of the sequence $\left( u_{k}\right) _{k}$which allows us by the lower
semicontinuity of the norm to write 
\begin{equation*}
\left\Vert u_{\epsilon }\right\Vert \leq \underset{m\rightarrow +\infty }{%
\lim }\inf \left\Vert u_{m}\right\Vert \text{.}
\end{equation*}%
And by the inequalities (\ref{2.16}), (\ref{2.22}) we deduce that the
sequence ($u_{\varepsilon }$)$_{\varepsilon }$ of the $\varepsilon $%
-approximating solutions is bounded in $H_{k}^{2}\left( M\right) $ for
sufficiently small $\epsilon >0$ i.e. 
\begin{equation}
\left\Vert u_{\epsilon }\right\Vert ^{2}\leq 4nC+1  \label{2.26}
\end{equation}%
thus we can extract a subsequence still labelled $\left( u_{m}\right) _{m}$
satisfying:

i) $u_{m}\longrightarrow u$ weakly in $H_{k}^{2}\left( M\right) $

ii) $u_{m}\longrightarrow u$ strongly in $L^{p}\left( M\right) $\ for $\
p<2^{\sharp }$

iii) $u_{m}\longrightarrow u$ a.e. in $M$.

vi) $u_{m}^{2^{\sharp }-1}\longrightarrow u^{2^{\sharp }-1}$weakly in $L^{%
\frac{2^{\sharp }}{2^{\sharp }-1}}.$\newline

Furthermore the sequence $(u_{m})_{m}$ is bounded below: indeed as the
functions $u_{k}$ are continuous, denote by $x_{m}$ their respective
maximums on $M$ $\ $and put $x_{o}=\lim x_{m}$ ( a subsequence of $\left(
x_{m}\right) _{m}$ still labelled $(x_{m})_{m}$ ). Since by assumption the
operator $P_{g}$ admits a positive Green function, then we can write%
\begin{equation*}
u_{m}(x_{m})=\int_{M}G\left( x_{m},y\right) \left( B\left( y\right) \left(
u_{m}^{+}\left( y\right) \right) ^{2^{\sharp }-1}+\frac{A\left( y\right)
u_{m}^{+}\left( y\right) }{\left( \varepsilon +\left( u_{m}^{+}\left(
y\right) \right) ^{2}\right) ^{2^{\flat }+1}}+\frac{C(y)u_{m}^{+}(y)}{\left(
\varepsilon +\left( u_{m}^{+}\left( y\right) \right) ^{2}\right) ^{\frac{p+1%
}{2}}}\right) dv_{g}
\end{equation*}%
and by Fatou's lemma, we get%
\begin{equation*}
\lim \inf_{m}u_{m}(x_{m})\geq
\end{equation*}%
\begin{equation*}
\int_{M}\lim \inf_{m}\left[ G\left( x_{m},y\right) \left( B\left( y\right)
\left( u_{m}^{+}\left( y\right) \right) ^{2^{\sharp }-1}+\frac{A\left(
y\right) u_{m}^{+}\left( y\right) }{\left( \varepsilon +\left(
u_{m}^{+}\left( y\right) \right) ^{2}\right) ^{2^{\flat }+1}}+\frac{%
C(y)u_{m}^{+}\left( y\right) }{\left( \varepsilon +\left( u_{m}^{+}\left(
y\right) \right) ^{2}\right) ^{\frac{p+1}{2}}}\right) \right] dv_{g}
\end{equation*}%
\begin{equation*}
=\int_{M}\lim \inf_{m}G\left( x_{m},y\right) \left( B\left( y\right) \left(
u^{+}\left( y\right) \right) ^{2^{\sharp }-1}+\frac{A\left( y\right)
u^{+}\left( y\right) }{\left( \varepsilon +\left( u^{+}\left( y\right)
\right) ^{2}\right) ^{2^{\flat }+1}}+\frac{C(y)u^{+}\left( y\right) }{\left(
\varepsilon +\left( u^{+}\left( y\right) \right) ^{2}\right) ^{\frac{p+1}{2}}%
}\right) dv_{g}.
\end{equation*}%
And since the functions $A$, $B$, $C$ are positive, then $\lim
\inf_{m}u_{m}(x_{m})=0$ implies that $u^{+}=0$. This contradicts relation (%
\ref{2.25}). Thus, there exists $\delta >0$, such that $u_{m}\geq \delta $.
We can once again use Lebesgue's dominated convergence theorem to get 
\begin{equation*}
\frac{1}{\left( \varepsilon _{m}+\left( u_{m}\right) ^{2}\right) ^{q}}%
\rightarrow \frac{1}{\left( u\right) ^{2q}}\text{ strongly in }L^{p}\left(
M\right) \text{, }\forall p\geq 1\text{, }\forall q\geq 1\text{.}
\end{equation*}%
Since for $m$ large enough $u_{m}>0$ there is $\tilde{\varepsilon}>0$ such
that 
\begin{equation*}
\text{\ \ \ }\frac{1}{\left( \varepsilon _{m}+u_{m}^{2}\right) ^{q}}\leq 
\frac{1}{\tilde{\varepsilon}^{q}}\text{ with }q>0\text{.}
\end{equation*}%
Thus by Lebesgue's dominated convergence theorem, we infer that%
\begin{equation*}
\frac{1}{\left( \varepsilon _{m}+\left( u_{m}\right) ^{2}\right) ^{q}}%
\longrightarrow \frac{1}{u^{2q}}\text{ strongly in }L^{p}\left( M\right) 
\text{, \ }\forall p\geq 1,\forall q>0.
\end{equation*}%
Finally, with ii), it follows that%
\begin{equation*}
\ \frac{u_{k}}{\left( \varepsilon _{m}+\left( u_{m}\right) ^{2}\right)
^{2^{\flat }+1}}\longrightarrow \frac{1}{u^{2^{\sharp }+1}}\text{\ strongly
in }L^{2}\left( M\right)
\end{equation*}%
with $u>0$. Letting $\varepsilon _{m}\rightarrow 0$ in (\ref{2.23}) as $%
m\rightarrow +\infty ,$we get that $u$ is a weak psitive solution of
equation (\ref{0.1}). By the same reasoning as that of the regularity of the
solution $u_{\varepsilon }$ of the equation (\ref{2.23}) we obtain that $%
u\in C^{2k,\alpha }\left( M\right) $ with $\alpha \in \left( 0,1\right) $.
Since $u>0$, the right-hand-side of (\ref{0.1}) has the same regularity as $%
u $ and by successive iterations we obtain that $u\in C^{\infty }\left(
M\right) $.
\end{proof}

\section{Existence of a second solution}

According to the previous section our functional admits a local maximum $%
C_{\varepsilon }$, this means the following inequalities 
\begin{equation*}
I_{\varepsilon }(t_{1}\varphi )<\Phi (t_{0})<I_{\varepsilon }(t_{0}\varphi
)\leq C_{\varepsilon }
\end{equation*}%
where $t_{0}$, $t_{1}$ are real numbers satisfying $0<t_{1}<t_{0}$ and $%
\varphi \in C^{\infty }(M)$ with $\varphi >0$ and $\left\Vert \varphi
\right\Vert =1$.

On the other hand and as $I_{\varepsilon }(t\varphi )$ tends to $-\infty $
as $t$ goes to $+\infty $, there is $t_{2}>>t_{0}$ such that $I_{\varepsilon
}(t_{2}\varphi )<0$. Now if we let $t$ and $\varepsilon $ tend both to $%
0^{+} $, the functional $I_{\varepsilon }$ goes to $+\infty $. Indeed, 
\begin{eqnarray*}
\lim\limits_{t\rightarrow 0^{+}}\left[ \lim\limits_{\varepsilon \rightarrow
0^{+}}I_{\varepsilon }(t.\varphi )\right] &=&\lim\limits_{t\rightarrow 0^{+}}%
\left[ I^{(1)}(t.\varphi )+I_{0}^{(2)}(t.\varphi )\right] \\
&=&\lim\limits_{t\rightarrow 0^{+}}\int\limits_{M}\left[ (t.\varphi
)P_{g}(t.\varphi )-\frac{2}{2^{\sharp }}B(x)(t.\varphi )^{2^{\sharp }}\right]
dv(g) \\
&+&\lim\limits_{t\rightarrow 0^{+}}\left[ \frac{1}{2^{\sharp }}%
\int\limits_{M}\frac{A(x)}{\left( t.\varphi \right) ^{2^{\sharp }}}dv(g)+%
\frac{1}{p-1}\int\limits_{M}\frac{C(x)}{\left( t.\varphi \right) ^{p-1}}%
dv(g).\right] \\
&=&+\infty .
\end{eqnarray*}%
and it follows that for $\varepsilon $ small enough, there is near $0$ a
real number $0<t^{\prime }<<t_{1}$ such that $I_{\varepsilon }(t^{\prime
}\varphi )>\Phi (t_{0}>I_{\varepsilon }\left( t_{1}\varphi \right) $. What
let us see, from this fact, that our function has a local lower bound. We
will give the necessary conditions for this lower bound to exist, then we
show by Ekeland's lemma that this lower bound is reached.

We will need the following version of the Ekeland's lemma (see \cite{11})

\begin{lemma}
Let $V$ be a Banach space, $J$ be a $C^{1}$ lower bounded function on a
closed subset $F$ of $V$ and $c=\inf_{F}J$. Let $u_{\varepsilon }$ $\in F$
such that $c\leq J(u_{\varepsilon })\leq c+\varepsilon .$ Then there is $%
\overline{u}_{\varepsilon }\in F$ such that%
\begin{equation*}
\left\{ 
\begin{array}{c}
c\leq J(\overline{u}_{\varepsilon })\leq c+\varepsilon \\ 
\left\Vert \overline{u}_{\varepsilon }-u_{\varepsilon }\right\Vert _{V}\leq 2%
\sqrt{\varepsilon } \\ 
\forall u\in F\text{, }u\neq \overline{u}_{\varepsilon }\text{, }J(u)-J(%
\overline{u}_{\varepsilon })+\sqrt{\varepsilon }\left\Vert u-\overline{u}%
_{\varepsilon }\right\Vert _{V}>0\text{.}%
\end{array}%
\right.
\end{equation*}%
If moreover, $\overline{u}_{\varepsilon }$\ is in the interior of $F$, then 
\begin{equation*}
\left\Vert DJ(\overline{u}_{\varepsilon })\right\Vert _{V^{\prime }}\leq 
\sqrt{\varepsilon }\text{.}
\end{equation*}
\end{lemma}

We can consider the sequence ( $u_{\varepsilon }$ )$_{\varepsilon }$ in the
interior of $F$. Indeed if $u_{\varepsilon }$ is on the border of $F$ then
by the continuity of $J$ there is $\overline{u}_{\varepsilon }$ belonging to
interior of $F$ such that $\left\vert J(\overline{u}_{\varepsilon
})-J(u_{\varepsilon })\right\vert <\varepsilon $. Which gives, for $%
\varepsilon $ \ sufficiently, $c-\varepsilon <J(\overline{u}_{\varepsilon
})<c+2\varepsilon $ and $J(u)-J(\overline{u}_{\varepsilon })+\sqrt{%
\varepsilon }\left\Vert u-\overline{u}_{\varepsilon }\right\Vert
_{V}=J(u)-J(u_{\varepsilon })+J(u_{\varepsilon })-J(\overline{u}%
_{\varepsilon })+\sqrt{\varepsilon }\left\Vert u-u_{\varepsilon
}+u_{\varepsilon }-\overline{u}_{\varepsilon }\right\Vert _{V}$

$\geq J(u)-J(u_{\varepsilon })-\varepsilon +\sqrt{\varepsilon }\left\Vert
u-u_{\varepsilon }\right\Vert _{V}-\sqrt{\varepsilon }\left\Vert
u_{\varepsilon }-\overline{u}_{\varepsilon }\right\Vert
_{V}>J(u)-J(u_{\varepsilon })+\sqrt{\varepsilon }\left\Vert u-u_{\varepsilon
}\right\Vert _{V}-2\varepsilon >0$. So we can speak about the differential $%
DJ(u_{\varepsilon })$.

Before beginning the proof of the Theorem \ref{theorem2}, we will establish
some preliminary lemmas.

\begin{lemma}
\label{lem3} Let $\theta >0$ such that 
\begin{equation}
\left( \frac{a}{2}\right) ^{\frac{2}{2^{\sharp }}}<\theta ^{2}<a^{\frac{2}{%
2^{\natural }}}  \label{3.1}
\end{equation}%
where 
\begin{equation*}
a=\frac{1}{\left( 2\left( n-k\right) \right) ^{\frac{2^{\natural }}{2}}}
\end{equation*}%
and put 
\begin{equation*}
t_{3}=\left( \frac{a}{8}\right) ^{\frac{1}{2^{\sharp }}}t_{0}
\end{equation*}%
then we have the following inequality 
\begin{equation}
\Phi (t_{3})>\frac{a}{8}\Phi (t_{0}).  \label{3.2}
\end{equation}
\end{lemma}

\begin{proof}
Since ($t_{1}$ being defined as in lemma \ref{lem1} ) 
\begin{equation*}
t_{3}=\left( \frac{a}{8}\right) ^{\frac{1}{2^{\sharp }}}t_{0}<\theta
t_{0}=t_{1}
\end{equation*}%
and 
\begin{equation}
\left( \frac{a}{8}\right) ^{\frac{2}{2^{\sharp }}}>\frac{a}{8}.  \label{3.3}
\end{equation}%
by (\ref{3.2}), we get 
\begin{align*}
\Phi (t_{3})& =\frac{1}{2}t_{3}^{2}-\left( S.\max_{M}B(x)\right) \frac{%
t_{3}^{2^{\sharp }}}{2^{\sharp }} \\
& =\frac{1}{2}\left[ \left( \frac{a}{8}\right) ^{\frac{1}{2^{\sharp }}}t_{0}%
\right] ^{2}-\frac{1}{2^{\sharp }}t_{0}^{2}\frac{a}{8} \\
& =\frac{n}{k}.\left[ \frac{1}{2}\left( \frac{a}{8}\right) ^{\frac{2}{%
2^{\sharp }}}-\ \frac{a}{8}.\frac{n-2k}{2n}\right] \frac{k}{n}t_{0}^{2}
\end{align*}%
Knowing by (\ref{2.4}) that 
\begin{equation*}
\frac{k}{n}t_{0}^{2}=\Phi (t_{0})
\end{equation*}%
we deduce 
\begin{align*}
\Phi (t_{3})& =\left[ \frac{n}{2k}\left( \left( \frac{a}{8}\right) ^{\frac{2%
}{2^{\sharp }}}-\frac{a}{8}\right) +\frac{a}{8}\right] \Phi (t_{0}) \\
& >\frac{a}{8}\Phi (t_{0}).
\end{align*}%
Where we used the inequality (\ref{3.3}) in the last line.
\end{proof}

\begin{lemma}
\label{lem4} Given a Riemannian compact manifold $(M,g)$ of dimension $%
n>2k,\;k\in \mathbb{N}^{\ast }$ and $3<p<2^{\sharp }+1$. \newline
If 
\begin{equation}
\int_{M}Q_{g}dv_{g}\neq k(n-1)\omega _{n}  \label{3.4}
\end{equation}%
where $\omega _{n}$ is the volume of the round sphere; then there is a
constant $\lambda ^{\ast }>0$ such that: $\forall \varepsilon \in \left]
0,\lambda ^{\ast }\right[ $ the following inequality take place 
\begin{equation}
\int\limits_{M}^{{}}\frac{A(x)}{\left( \varepsilon +(t_{3}.\varphi
)^{2}\right) ^{\frac{2^{\sharp }}{2}}}dv_{g}\geq \frac{8-a}{a}\left(
\int\limits_{M}^{{}}\sqrt{A(x)}dv_{g}\right) ^{2}\frac{(n-2k)Q_{\tilde{g}}}{%
2t_{0}^{2^{\sharp }}}  \label{3.5}
\end{equation}

where $t_{3},a$ are chosen as in lemma \ref{lem3} .
\end{lemma}

\begin{proof}
Let $\varphi \in C^{\infty }(M),\varphi >0$ in $M$ with $\Vert \varphi \Vert
=1$. Put 
\begin{equation}
\beta _{1}=\left[ \left( \frac{8}{8-a}\right) ^{\frac{2}{2^{\sharp }}}-1%
\right] \frac{\Omega _{1}}{V(M)^{\frac{2}{2^{\sharp }}}}  \label{3.6}
\end{equation}%
and 
\begin{equation}
\beta _{2}=\left[ \left( \frac{1}{a}\right) ^{\left( \frac{2}{2^{\sharp }}%
\right) ^{2}}-1\right] \frac{\Omega _{2}}{V(M)^{\frac{2}{2^{\sharp }}}}
\end{equation}%
where $V(M)$ denotes the volume of $M$ and 
\begin{equation*}
\Omega _{1}=\left( \dfrac{2a}{8(n-2k)Q_{\tilde{g}}}\right) ^{\frac{2}{%
2^{\sharp }}}t_{0}^{2}.
\end{equation*}%
\begin{equation*}
\Omega _{2}=\left( \frac{2\left( a^{\frac{2}{2^{\sharp }}+1}\right) t_{0}^{2}%
}{(n-2k)Q_{\tilde{g}}}\right) ^{\frac{2}{2^{\sharp }}}t_{0}^{2}.
\end{equation*}%
Let 
\begin{equation*}
\lambda ^{\ast }=\min \left( \beta _{1},\beta _{2}\right) .
\end{equation*}%
By H\"{o}lder's inequality, we get:%
\begin{equation}
\left( \int\limits_{M}^{{}}\sqrt{A(x)}dv_{g}\right) ^{2}\leq I\left[ \Vert
\varepsilon +(t_{3}.\varphi )^{2}\Vert _{\frac{2^{\sharp }}{2}}\right] ^{%
\frac{2^{\sharp }}{2}}  \label{3.7}
\end{equation}%
where 
\begin{equation*}
I=\int\limits_{M}^{{}}\frac{A(x)}{\left( \varepsilon +(t_{3}.\varphi
)^{2}\right) ^{\frac{2^{\sharp }}{2}}}dv_{g}
\end{equation*}%
Independently, the Minkowski's inequality can be written 
\begin{equation*}
\Vert \varepsilon +(t_{3}.\varphi )^{2}\Vert _{\frac{2^{\sharp }}{2}}\leq
\Vert \varepsilon \Vert _{\frac{2^{\sharp }}{2}}+t_{3}^{2}\Vert \varphi
^{2}\Vert _{\frac{2^{\sharp }}{2}}
\end{equation*}%
consequently 
\begin{equation}
\left[ \Vert \varepsilon +(t_{3}.\varphi )^{2}\Vert _{\frac{2^{\sharp }}{2}}%
\right] ^{\frac{2^{\sharp }}{2}}\leq \left( \Vert \varepsilon \Vert _{\frac{%
2^{\sharp }}{2}}+t_{3}^{2}\Vert \varphi ^{2}\Vert _{\frac{2^{\sharp }}{2}%
}\right) ^{\frac{2^{\sharp }}{2}}  \label{3.8}
\end{equation}%
Notice that 
\begin{equation*}
\Vert \varepsilon \Vert _{\frac{2^{\sharp }}{2}}=\varepsilon .\left[ V(M)%
\right] ^{\frac{2}{2^{\sharp }}}
\end{equation*}%
and 
\begin{equation*}
\Vert \varphi ^{2}\Vert _{\frac{2^{\sharp }}{2}}=\Vert \varphi \Vert
_{2^{\sharp }}^{2}.
\end{equation*}%
From the conformal rule (\ref{0.0}) of the GJMS operator $P_{g}$ we have 
\begin{equation*}
P_{g}(\varphi )=\frac{n-2k}{2}Q_{\tilde{g}}.\varphi ^{2^{\sharp }-1}
\end{equation*}%
after multiplication by $\varphi $ and integration over the manifold $M$, we
get 
\begin{equation*}
\Vert \varphi \Vert ^{2}=\int_{M}\varphi P_{g}(\varphi )dv_{g}=\frac{n-2k}{2}%
\int_{M}Q_{\tilde{g}}.\varphi ^{2^{\sharp }}dv_{g}.
\end{equation*}%
Now, by the work done in \cite{3} and under the condition in (\ref{3.4}) we
can do a conformal change of the metric $g$ to a new metric $\tilde{g}$ such
that $Q_{\tilde{g}}$ is a constant which we suppose positive, hence 
\begin{equation*}
\Vert \varphi \Vert ^{2}=\frac{n-2k}{2}Q_{\tilde{g}}\Vert \varphi \Vert
_{2^{\sharp }}^{2^{\sharp }}
\end{equation*}%
since $\Vert \varphi \Vert =1$ we get 
\begin{equation}
\Vert \varphi \Vert _{2^{\sharp }}^{2}=\left( \dfrac{2}{(n-2k)Q_{\tilde{g}}}%
\right) ^{\frac{2}{2^{\sharp }}}  \label{3.9}
\end{equation}%
and therefore (\ref{3.8}) becomes 
\begin{equation*}
\left( \Vert \varepsilon +(t_{3}.\varphi )^{2}\Vert _{\frac{2^{\sharp }}{2}%
}\right) ^{\frac{2^{\sharp }}{2}}\leq \left( \varepsilon .\left[ V(M)\right]
^{\frac{2}{2^{\sharp }}}+t_{3}^{2}.\left( \dfrac{2}{(n-2k)Q_{\tilde{g}}}%
\right) ^{\frac{2}{2^{\sharp }}}\right) ^{\frac{2^{\sharp }}{2}}.
\end{equation*}%
\newline
Taking account of 
\begin{equation*}
t_{3}=\left( \frac{a}{8}\right) ^{\frac{1}{2^{\sharp }}}t_{0}
\end{equation*}%
(\ref{3.7}) is written as

\begin{equation*}
\left( \int\limits_{M}^{{}}\sqrt{A(x)}dv_{g}\right) ^{2}\leq I\left(
\varepsilon V(M)^{\frac{2}{2^{\sharp }}}+\left( \frac{a}{8}\right) ^{\frac{2%
}{2^{\sharp }}}t_{0}^{2}.\left( \dfrac{2}{(n-2k)Q_{\tilde{g}}}\right) ^{%
\frac{2}{2^{\sharp }}}\right) ^{\frac{2^{\sharp }}{2}}\text{.}
\end{equation*}%
And since $0<\varepsilon <\lambda ^{\ast }\leq \beta _{1}$, we get

\begin{eqnarray*}
\left( \int\limits_{M}^{{}}\sqrt{A(x)}dv_{g}\right) ^{2} &\leq &I.\left( %
\left[ \left( \frac{8}{8-a}\right) ^{\frac{2}{2^{\sharp }}}-1\right]
\Omega_1 +\Omega_1 \right) ^{\frac{2^{\sharp }}{2}} \\
&\leq &I\left( \frac{8}{8-a}\right) \Omega_1 ^{\frac{2^{\sharp }}{2}} \\
&\leq &I\left( \frac{8}{8-a}\right) \left( \left( \dfrac{2a}{8(n-2k)Q_{%
\tilde{g}}}\right) ^{\frac{2}{2^{\sharp }}}t_{0}^{2}\right) ^{\frac{%
2^{\sharp }}{2}}\text{.}
\end{eqnarray*}

Finally, we deduce 
\begin{equation*}
I\geq \left( \int\limits_{M}^{{}}\sqrt{A(x)}dv_{g}\right) ^{2}\left( \frac{%
8-a}{a}\right) \frac{(n-2k)Q_{\tilde{g}}}{2.t_{0}^{2^{\sharp }}}.
\end{equation*}
\end{proof}

Now we are able to prove the existence of a second solution to equation (\ref%
{0.1}), that is to say the proof of theorem \ref{theorem2}.

\begin{proof}
The proof will be done in four steps. \newline
\textit{1}$^{\text{st}}$\textit{\ step.} The functional $I_{\varepsilon }$
has a local lower bound. \newline
This consists to find a strictly positive real number $\lambda ^{\ast }$
such that $\forall \varepsilon \in \left] 0,\lambda ^{\ast }\right[ $ one
has the following inequality

\begin{equation*}
I_{\varepsilon }(t_{3}\varphi )>\Phi (t_{0})\quad \forall \varphi \in
C^{\infty }(M),\;\Vert \varphi \Vert =1,\text{ with }t_{3}<t_{1}.
\end{equation*}%
Indeed, according to Lemma \ref{lem3} inequality (\ref{3.2}) and inequality (%
\ref{2.2}); one has%
\begin{align}
I_{\varepsilon }(t_{3}\varphi )& =I^{(1)}(t_{3}\varphi
)+I^{(2)}(t_{3}\varphi )  \label{3.10} \\
& >\frac{a}{8}\Phi (t_{0})+\frac{1}{2^{\sharp }}\int_{M}\frac{A\left(
x\right) }{\left( \varepsilon +(t_{3}\varphi )^{2}\right) ^{2^{\flat }}}%
dv_{g}  \notag \\
& +\frac{1}{p-1}\int_{M}\frac{C(x)}{\left( \varepsilon +(t_{3}\varphi
)^{2}\right) ^{\frac{p-1}{2}}}dv_{g}
\end{align}

and as by assumption 
\begin{equation*}
\left( \int\limits_{M}^{{}}\sqrt{A(x)}dv_{g}\right) ^{2}\left( \frac{8-a}{a}%
\right) \frac{(n-2k)Q_{\tilde{g}}}{2.t_{0}^{2^{\sharp }}}>2^{\sharp }\frac{k%
}{n}t_{0}^{2}(1-\frac{a}{8})
\end{equation*}%
and 
\begin{equation*}
\lambda ^{\ast }=\min (\beta _{1},\beta _{2})
\end{equation*}%
knowing that 
\begin{equation*}
\Phi (t_{0})=\frac{k}{n}.t_{0}^{2}
\end{equation*}%
it follows by Lemma \ref{lem4} that, $\forall \varepsilon \in \left]
0,\lambda ^{\ast }\right[ $ 
\begin{equation}
\frac{1}{2^{\sharp }}\int\limits_{M}^{{}}\frac{A(x)}{\left( \varepsilon
+(t_{3}\varphi )^{2}\right) ^{\frac{2^{\sharp }}{2}}}dv_{g}>\left( 1-\frac{a%
}{8}\right) \Phi (t_{0}).  \label{3.11}
\end{equation}

Finally, by combination of (\ref{3.10}), (\ref{3.11}) and the fact that the
function $C>0$, we get

\begin{eqnarray*}
I_{\varepsilon }(t_{3}\varphi ) &>&\frac{a}{8}\Phi (t_{0})+\left( 1-\frac{a}{%
8}\right) \Phi (t_{0}) \\
&>&\Phi (t_{0})\text{.}
\end{eqnarray*}
Hence our result. \newline

\textit{2}$^{nd}$ \textit{step}. The infimum of the functional $%
I_{\varepsilon }$ is reached.

Denote by $\overline{B}(0,t_{1})=\left\{ u\in H_{k}^{2}(M):\left\Vert
u\right\Vert \leq t_{1}\right\} $ the closed ball centred at the origin $0$
of radius $t_{1}$ in $H_{k}^{2}(M)$. In this section we will show that $%
c_{\varepsilon }=\inf_{B(0,t_{1})}I_{\varepsilon }$\ \ ( $c_{\varepsilon
}<\Phi (t_{0})$ )\ is reached. By Ekeland's Lemma, there exists a sequence $%
\left( u_{m}\right) _{m\in N}$ in $B(0,t_{1})$ such that $I_{\varepsilon
}(u)\rightarrow c_{\varepsilon }=Inf_{B(0,t_{1})}I_{\varepsilon }$ and $%
DI_{\varepsilon }(u_{m})\rightarrow 0$ strongly in the dual space of $%
H_{k}^{2}(M)$. That is to say ($u_{m}$) is a Palais-Smale sequence, so by
the same arguments as in Theorem \ref{theorem2} and Theorem\ref{theorem1},
we get that equation (\ref{0.1}) has a smooth\ positive solution $v$. Since
the $\varepsilon $-approximating solutions are obtained as weak limit of
sequences of functions from $\overline{B}(0,t_{1})$, it follows by the weak
lower semi-continuity of the norm that these $\varepsilon $-approximating
solutions are in $\overline{B}(0,t_{1}).$ As in turn $v$ is obtained as a
limit of a sequence of $\varepsilon $-approximating solutions that $v\in 
\overline{B}(0,t_{1}).$

\textit{3}$^{rd}$\textit{-step}. The two solutions are distinct.

To show that the two solutions $u$ and $v$ are different, we will verify
that their respective energies are different.

Put $t_{4}=a^{\frac{1}{2^{\sharp }}}t_{0}$, it is clear that $t_{4}>t_{1}$ (
see the assumptions of Lemma \ref{lem1}) and since $\left\Vert v\right\Vert
\leq t_{1}$ then if $\left\Vert u\right\Vert \geq t_{4}$, $u\neq v$. So we
may suppose that $\left\Vert u\right\Vert <t_{4}$.

Imitating the computations made in the previous section and taking into
account that in this time we take $\epsilon \leq \beta _{2}$, it is not hard
to get 
\begin{equation*}
\int_{M}\frac{A(x)}{\left( \epsilon +\left( t_{4}u\right) ^{2}\right) ^{%
\frac{2^{\sharp }}{2}}}dv_{g}\geq \frac{\left( n-2k\right) Q_{\tilde{g}}}{%
2.a.t_{0}^{2^{\sharp }+2}}\left( \int_{M}\sqrt{A(x)}dv_{g}\right) ^{2}.
\end{equation*}%
Now, since it is easy to see that%
\begin{equation*}
\frac{1}{2^{\sharp }}\int_{M}\frac{A(x)}{u^{2^{\sharp }}}dv_{g}\geq \frac{%
t_{4}^{2^{\sharp }}}{2^{\sharp }}\int_{M}\frac{A(x)}{\left( \epsilon +\left(
t_{4}u\right) ^{2}\right) ^{\frac{2^{\sharp }}{2}}}dv_{g}
\end{equation*}%
and by the hypothesis (\ref{0.8}) of Theorem \ref{theorem2}, we infer that 
\begin{equation*}
\frac{1}{2^{\sharp }}\int_{M}\frac{A(x)}{u^{2^{\sharp }}}dv_{g}\geq \frac{%
a.t_{0}^{2^{\sharp }-2}}{8}\Phi (t_{0})
\end{equation*}%
where $\Phi (t_{0})=\frac{k}{n}t_{0}^{2}$.

Now, we will estimate the energy of the solution $u$

\begin{equation*}
I(u)=\frac{1}{2}\left\Vert u\right\Vert ^{2}-\frac{1}{2^{\sharp }}%
\int_{M}B(x)u^{2^{\sharp }}dv_{g}+\frac{1}{2^{\sharp }}\int_{M}\frac{A(x)}{%
u^{2^{\sharp }}}dv_{g}+\frac{1}{\left( p-1\right) }\int_{M}\frac{C(x)}{%
u^{p-1}}dv_{g}
\end{equation*}%
and since 
\begin{equation*}
\left\Vert u\right\Vert ^{2}=\int_{M}B(x)u^{2^{\sharp }}dv_{g}+\int_{M}\frac{%
A(x)}{u^{2^{\sharp }}}dv_{g}+\int_{M}\frac{C(x)}{u^{p-1}}dv_{g}
\end{equation*}%
we deduce that 
\begin{eqnarray*}
I(u) &=&\frac{k}{n}\left\Vert u\right\Vert ^{2}+\left( \frac{1}{2^{\sharp }}+%
\frac{1}{2^{\sharp }}\right) \int_{M}\frac{A(x)}{u^{2^{\sharp }}}%
dv_{g}+\left( \frac{1}{2^{\sharp }}+\frac{1}{\left( p-1\right) }\right)
\int_{M}\frac{C(x)}{u^{p-1}}dv_{g} \\
&\geq &\frac{2}{2^{\sharp }}\int_{M}\frac{A(x)}{u^{2^{\sharp }}}dv_{g} \\
&\geq &2\frac{a.t_{0}^{2^{\sharp }-2}}{8}\Phi (t_{0}).
\end{eqnarray*}%
and taking into account of the value $a=\frac{1}{\left( 2(n-k)\right)
^{2^{\sharp }}}$ and the fact that 
\begin{equation*}
0<\left( S\underset{x\in M}{\max }B(x)\right) <\frac{a}{4}
\end{equation*}%
we infer%
\begin{equation*}
I(u)>\Phi \left( t_{0}\right) \text{.}
\end{equation*}%
Since the energy $I(v)$ of the solution $v$ is less than $\Phi \left(
t_{0}\right) $ we conclude that $u\neq v$.

\textit{4}$^{\text{th}}$\textit{-step}. The conditions of the theorem
intersect.

Indeed, let us rewrite the condition (\ref{0.8}) of Theorem \ref{theorem2} 
\begin{equation}
\frac{1}{2^{\sharp }}\left( \int\limits_{M}^{{}}\sqrt{A(x)}dv_{g}\right)
^{2}>\frac{k}{8n}t_{0}^{2+2^{\sharp }}\frac{2a}{(n-2k)Q_{\tilde{g}}}.
\label{3.12}
\end{equation}%
By H\"{o}lder inequality, we get%
\begin{eqnarray}
\int\limits_{M}^{{}}\sqrt{A(x)}dv_{g} &=&\int_{M}\sqrt{\dfrac{A(x)}{\varphi
^{2^{\sharp }}}}\varphi ^{\frac{2^{\sharp }}{2}}dv_{g}  \notag \\
&\leq &\left( \int_{M}\dfrac{A(x)}{\varphi ^{2^{\sharp }}}dv_{g}\right) ^{%
\frac{1}{2}}\left( \int_{M}\varphi ^{2^{\sharp }}dv_{g}\right) ^{\frac{1}{2}}
\notag \\
&\leq &\Vert \varphi \Vert _{2^{\sharp }}^{\frac{2^{\sharp }}{2}}\left(
\int_{M}\dfrac{A(x)}{\varphi ^{2^{\sharp }}}dv_{g}\right) ^{\frac{1}{2}}
\end{eqnarray}%
From equality (2.10) in the proof of Lemma 2.3 and the fact that $\Vert
\varphi \Vert =1$, it comes that 
\begin{equation*}
\frac{1}{2^{\sharp }}\left( \int\limits_{M}^{{}}\sqrt{A(x)}dv_{g}\right)
^{2}\leq \frac{1}{2^{\sharp }}\left( \dfrac{2}{(n-2k)Q_{\tilde{g}}}\right)
\int_{M}\dfrac{A(x)}{\varphi ^{2^{\sharp }}}dv_{g}
\end{equation*}%
with the condition (\ref{0.6}) of Theorem \ref{theorem2}, we obtain 
\begin{equation*}
\frac{1}{2^{\sharp }}\left( \int\limits_{M}^{{}}\sqrt{A(x)}dv_{g}\right)
^{2}\leq \left( \dfrac{2}{(n-2k)Q_{\tilde{g}}}\right) C(n,p,k)\left( S%
\underset{x\in M}{\max }B(x)\right) ^{\frac{2+2^{\sharp }}{2-2^{\sharp }}}%
\text{.}
\end{equation*}%
Since $3<p<2^{\sharp }+1$, we may take $C(n,p,k)=C_{1}(n,p,k)$ and we have
also $\theta ^{2^{\sharp }}\leq \theta ^{p-1}$ and therefore 
\begin{equation}
\frac{1}{2^{\sharp }}\left( \int\limits_{M}^{{}}\sqrt{A(x)}dv_{g}\right)
^{2}\leq \left( \dfrac{2}{(n-2k)Q_{\tilde{g}}}\right) \frac{2k-1}{4n}\theta
^{2^{\sharp }}t_{0}^{2+2^{\sharp }}.  \label{3.13}
\end{equation}%
By combining conditions (\ref{3.12}) and (\ref{3.13}), we get the following
double inequality 
\begin{equation*}
\frac{k}{8n}t_{0}^{2+2^{\sharp }}\frac{2a}{(n-2k)Q_{\tilde{g}}}<\frac{1}{%
2^{\sharp }}\left( \int\limits_{M}^{{}}\sqrt{A(x)}dv_{g}\right) ^{2}\leq
\left( \dfrac{2}{(n-2k)Q_{\tilde{g}}}\right) \frac{2k-1}{4n}\theta
^{2^{\sharp }}t_{0}^{2+2^{\sharp }}.
\end{equation*}%
Which in turn is equivalent to 
\begin{equation*}
k\frac{a}{2}<\frac{2n(n-2k)Q_{\tilde{g}}t_{0}^{-2-2^{\sharp }}}{2^{\sharp }}%
\left( \int\limits_{M}^{{}}\sqrt{A(x)}dv_{g}\right) ^{2}\leq \left(
2k-1\right) \theta ^{2^{\sharp }}
\end{equation*}%
and since $\frac{a}{2}<\theta ^{2^{\sharp }}$ we get%
\begin{equation*}
1<\frac{4n(n-2k)Q_{\tilde{g}}t_{0}^{-2-2^{\sharp }}}{2^{\sharp }a.k}\left(
\int\limits_{M}^{{}}\sqrt{A(x)}dv_{g}\right) ^{2}\leq \left( 2-\frac{1}{k}%
\right)
\end{equation*}%
with $k\geq 1$. The smooth functions\ $A$ that fulfill the assumptions (0.6)
and (0.8) of theorem \ref{theorem2} are those that satisfy the following
double inequality 
\begin{equation*}
C<A\leq \left( 2-\frac{1}{k}\right) C
\end{equation*}%
where \ $C=\frac{2^{\sharp }a.k}{4n(n-2k)Q_{\tilde{g}}t_{0}^{-2-2^{\sharp
}}V(M)^{2}}$ and $V(M)$ is the volume of $M$.
\end{proof}

\section{Nonexistence of solution}

In this section we will be placed in a closed ball $\overline{B}(0,R)$ of $%
H_{k}^{2}(M)$ centered at the origin $0$ and of radius $R>0$, we prove that
under some condition (inequality \ref{0.10} of Theorem \ref{theorem3}) that
the equation (\ref{0.1}) has no solution i.e. we will give the proof of
Theorem \ref{theorem3}.

\begin{proof}
( Proof of \ Theorem \ref{theorem3}) Suppose that there exists a smooth
positive solution $u\in H_{k}^{2}(M)$ such that $\Vert u\Vert
_{H_{k}^{2}(M)}\leq R$. By multiplying both sides of equation (\ref{2.1}) by 
$u$ end integrating over $M$, we get 
\begin{equation*}
\int_{M}uP_{g}(u)dv_{g}=\int_{M}\left( B\left( x\right) u^{2^{\sharp }}+%
\frac{A\left( x\right) }{u^{2^{\sharp }}}+\frac{C\left( x\right) }{u^{p-1}}%
\right) dv_{g}.
\end{equation*}%
And since $\Vert u\Vert _{p_{g}}=\sqrt{\int\limits_{M}uP_{g}(u)dv_{g}}$ is a
norm equivalent to $\Vert u\Vert _{H_{k}^{2}(M)}$, there exists a constant $%
S>0$ such that 
\begin{equation*}
\Vert u\Vert \leq S\Vert u\Vert _{H_{k}^{2}(M)}\text{.}
\end{equation*}%
Then it follows that 
\begin{equation}
\int_{M}\left( B\left( x\right) u^{2^{\sharp }}+\frac{A\left( x\right) }{%
u^{2^{\sharp }}}+\frac{C\left( x\right) }{u^{p-1}}\right) dv_{g}\leq \left(
SR\right) ^{2}\text{.}  \label{4.1}
\end{equation}%
Moreover H\"{o}lder's inequality allows us to write 
\begin{equation}
\int_{M}\sqrt{B\left( x\right) C\left( x\right) }dv_{g}\leq \left(
\int\limits_{M}\frac{C(x)}{u^{p-1}}dv_{g}\right) ^{\frac{1}{2}}\left(
\int\limits_{M}B(x)u^{p-1}dv_{g}\right) ^{\frac{1}{2}}\text{.}
\end{equation}%
By applying again the H\"{o}lder's inequality, one finds 
\begin{eqnarray}
\int_{M}B\left( x\right) u^{p-1}dv_{g} &=&\int_{M}B\left( x\right) ^{1-\frac{%
p-1}{2^{\sharp }}}\left( B\left( x\right) ^{\frac{1}{2^{\sharp }}}u\right)
^{p-1}dv_{g}  \notag \\
&\leq &\left( \int_{M}B\left( x\right) dv_{g}\right) ^{\frac{2^{\sharp }-p+1%
}{2^{\sharp }}}\left( \int_{M}\left[ \left( B\left( x\right) ^{\frac{1}{%
2^{\sharp }}}u\right) ^{p-1}\right] ^{\frac{2^{\sharp }}{p-1}}dv_{g}\right)
^{\frac{p-1}{2^{\sharp }}}  \notag \\
&\leq &\left( \int_{M}B\left( x\right) dv_{g}\right) ^{\frac{2^{\sharp }-p+1%
}{2^{\sharp }}}\left( \int_{M}B\left( x\right) u^{2^{\sharp }}dv_{g}\right)
^{\frac{p-1}{2^{\sharp }}}.
\end{eqnarray}%
So 
\begin{equation*}
\left( \int_{M}\sqrt{B\left( x\right) C\left( x\right) }dv_{g}\right)
^{2}\leq \left( \int_{M}B\left( x\right) dv_{g}\right) ^{\frac{2^{\sharp
}-p+1}{2^{\sharp }}}.\left( \int_{M}B\left( x\right) u^{2^{\sharp
}}dv_{g}\right) ^{\frac{p-1}{2^{\sharp }}}\int_{M}\frac{C(x)}{u^{p-1}}dv_{g}%
\text{.}
\end{equation*}%
Letting 
\begin{equation*}
D=\left( \int_{M}\sqrt{B\left( x\right) C\left( x\right) }dv_{g}\right)
^{2}\left( \int_{M}B\left( x\right) dv_{g}\right) ^{\frac{p-2^{\sharp }-1}{%
2^{\sharp }}}
\end{equation*}%
it comes 
\begin{equation*}
\int_{M}\frac{C(x)}{u^{p-1}}dv_{g}\geq D\left( \int_{M}B\left( x\right)
u^{2^{\sharp }}dv_{g}\right) ^{\frac{1-p}{2^{\sharp }}}.
\end{equation*}%
and therefore (\ref{4.1}) becomes 
\begin{equation*}
(SR)^{2}\geq \int_{M}B\left( x\right) u^{2^{\sharp }}dv_{g}+\int_{M}\frac{%
A\left( x\right) }{u^{2^{\sharp }}}dv_{g}+D\left( \int_{M}B\left( x\right)
u^{2^{\sharp }}dv_{g}\right) ^{\frac{1-p}{2^{\sharp }}}\text{.}
\end{equation*}%
Since $A$ is of positive values, then 
\begin{equation*}
\int_{M}\frac{A\left( x\right) }{u^{2^{\sharp }}}dv_{g}\geq 0
\end{equation*}%
so it comes that 
\begin{equation*}
(SR)^{2}\geq \int_{M}B\left( x\right) u^{2^{\sharp }}dv_{g}+D\left(
\int_{M}B\left( x\right) u^{2^{\sharp }}dv_{g}\right) ^{\frac{1-p}{2^{\sharp
}}}
\end{equation*}%
and if we set 
\begin{equation*}
t=\int_{M}B\left( x\right) u^{2^{\sharp }}dv_{g}
\end{equation*}%
we obtain 
\begin{equation*}
(RS)^{2}\geq f(t)
\end{equation*}%
where 
\begin{equation*}
f(t)=t+Dt^{\frac{1-p}{2^{\sharp }}}.
\end{equation*}%
$f$ \ has a minimum at 
\begin{equation*}
t_{0}=\left( \frac{p-1}{2^{\sharp }}D\right) ^{\frac{2^{\sharp }}{2^{\sharp
}+p-1}}
\end{equation*}%
and consequently 
\begin{equation*}
\forall t>0\text{, }f(t)\geq \min_{t>0}f(t)=f(t_{0})=\frac{2^{\sharp }+p-1}{%
p-1}\left( \frac{p-1}{2^{\sharp }}D\right) ^{\frac{2^{\sharp }}{2^{\sharp
}+p-1}}\text{.}
\end{equation*}%
Finally, replacing $D$ by its value, we obtain 
\begin{equation*}
(RS)^{2}\geq \frac{2^{\sharp }+p-1}{p-1}\left( \frac{p-1}{2^{\sharp }}%
\right) ^{\frac{2^{\sharp }}{2^{\sharp }+p-1}}\left( \int_{M}\sqrt{B\left(
x\right) C\left( x\right) }dv_{g}\right) ^{2.\frac{2^{\sharp }}{2^{\sharp
}+p-1}}\left( \int_{M}B\left( x\right) dv_{g}\right) ^{\frac{p-1-2^{\sharp }%
}{2^{\sharp }+p-1}}\text{.}
\end{equation*}%
and if we set 
\begin{equation*}
C(n,p,k)=\frac{2^{\sharp }+p-1}{p-1}\left( \frac{p-1}{2^{\sharp }}\right) ^{%
\frac{2^{\sharp }}{2^{\sharp }+p-1}}
\end{equation*}%
then it comes that 
\begin{equation*}
(RS)^{2}\geq C(n,p,k)\left( \frac{\int_{M}\sqrt{B\left( x\right) C\left(
x\right) }dv_{g}}{\int_{M}B\left( x\right) dv_{g}}\right) ^{2.\frac{%
2^{\sharp }}{2^{\sharp }+p-1}}\int_{M}B\left( x\right) dv_{g}.
\end{equation*}
\end{proof}

\end{document}